\documentclass[12pt]{amsart}

\usepackage{amsmath}
\usepackage{amsthm}
\usepackage{amssymb}
\usepackage{amsfonts,mathrsfs}
\usepackage{stmaryrd}
\usepackage{amsxtra}
\usepackage{epsfig}
\usepackage{verbatim}
\usepackage{color}
\usepackage{hyperref}

\theoremstyle{plain}
\newtheorem{theorem}[equation]{Theorem}
\newtheorem{proposition}[equation]{Proposition}
\newtheorem{lemma}[equation]{Lemma}

\theoremstyle{remark}
\newtheorem{remark}{Remark}
\newcommand{\dbar}{\overline{\partial}}
\newcommand{\dbarstar}{\overline{\partial}^{\ast}}
\newcommand{\Cn}{\mathbb{C}^n}
\newcommand{\BOm}{\mathbf{B}_{\Omega}}
\newcommand{\BOmg}{\mathbf{B}_{\Omega_g}}

\newcommand{\Om}{\Omega}

\renewcommand\Re{\operatorname{Re}}
\begin{document}

\title[Regularity and Nebenh\"ulle on Hartogs Domains]
{Regularity of Canonical Operators and Nebenh\"ulle: Hartogs Domains}
\author{Yunus E. Zeytuncu}
\address{Texas A\&M University, Department of
Mathematics, College Station, TX 77843
\textit{Current Address:} Department of Mathematics and Statistics, University of Michigan-Dearborn, Dearborn, MI 48128}

\email{zeytuncu@math.tamu.edu, zeytuncu@umich.edu}
\begin{abstract}
We relate the regularity of the Bergman projection operator and the canonical solution operator to the Nebenh\"ulle of complete Hartogs domains. 
\end{abstract}

\maketitle
\section{Introduction}
\subsection{Definitions}
In this section, we present definitions and notations for the terms that are used in the paper. The reader can refer to \cite{StraubeBook} for the details of these and other definitions. 

Let $\Om$ be a bounded pseudoconvex domain in $\mathbb{C}^n$ and let $L^2_{(0,q)}(\Om)$ be the space of $(0,q)$-forms on $\Om$ with square integrable coefficients (for $(0,0)$-forms, i.e., functions, no subscript will be used). Each
such form can be written uniquely as 
$$u=\sideset{}{'}\sum_{J}u_J d\overline{z_J}$$
where $J=(j_1,\dots, j_q)$ is a strictly increasing multi-index, $\sideset{}{'}\sum$ denotes the summation over such indices, and $d\overline{z_J}=d\overline{z_{j_1}}\wedge\dots\wedge d\overline{z_{j_q}}$.

We define the following inner product on $L^2_{(0,q)}(\Om)$,
$$(u,v)=\left(\sideset{}{'}\sum_{J}u_J d\overline{z_J},\sideset{}{'}\sum_{J}v_J d\overline{z_J}\right)=\sideset{}{'}\sum_{J}\int_{\Om}u_J \overline{v_J}dV,$$
under which $L^2_{(0,q)}(\Om)$ is a Hilbert space. We also define the standard $\dbar$-operator (the Cauchy-Riemann operator) on $(0,q)$-forms as
$$\dbar\left(\sideset{}{'}\sum_{J}u_J d\overline{z_J}\right)=\sum_{j=1}^{n}\sideset{}{'}\sum_{J}\frac{\partial u_J}{\partial \overline{z_j}} d\overline{z_j}\wedge d\overline{z_J},$$
where the derivatives are computed as distributional derivatives. We say a form $u\in L^2_{(0,q)}(\Om)$ is in the domain of $\dbar$ if $\dbar u \in L^2_{(0,q+1)}(\Om)$. In this standard setup, the operator $\dbar$ is a closed and densely defined operator from $L^2_{(0,q)}(\Om)$ to $L^2_{(0,q+1)}(\Om)$. Moreover, it has a Hilbert space adjoint that is denoted by $\dbar^{\ast}$.

We define the complex Laplacian (also referred to as the $\dbar$-Neumann Laplacian) on $(0,q)$-forms as
\begin{equation}\label{complexlaplacian}
\Box=\Box_{q}=\dbar\dbar^{\ast}+\dbar^{\ast}\dbar, 
\end{equation}
where each operator is defined at the correct form level and with domain so that the compositions are defined. It is clear that Dom($\Box$) involves two boundary conditions: $u\in$ Dom($\dbar^{\ast}$) and $\dbar u\in$ Dom($\dbar^{\ast}$). The first one is a Dirichlet condition and the second 
one is a Neumann condition.

It is known that $\Box$ has a bounded inverse (that is a solution operator) on bounded pseudoconvex domains. This operator is called the $\dbar$-Neumann operator of
$\Om$ and it is denoted by $N=N_q$ (see also \cite{StraubeBook}).

For $\alpha \in L^2_{(0,q)}(\Om)$ such that $\dbar \alpha=0$, the $\dbar$-problem is to find a form $u \in L^2_{(0,q-1)}(\Om)$ such that 
\begin{equation}\label{dbar}
\dbar u =\alpha.
\end{equation}
By using the machinery above we note that $\dbarstar N \alpha$ is a solution for \eqref{dbar} and moreover this solution has the smallest $L^2$-norm among all the solutions. The operator $\dbarstar N$ on $L^2_{(0,q)}(\Om)$ is called the canonical solution operator.

A function $f\in C^{\infty}(\Om)$ is said to be holomorphic on $\Om$, if $\dbar f=0$ in $\Om$. Let $\mathcal{O}(\Om)$ denote the set of holomorphic functions on $\Om$ and $L^2_a(\Om)$ denote 
the intersection $\mathcal{O}(\Om)\cap L^2(\Om)$. It is a consequence of the Cauchy integral formula that $L^2_a(\Omega)$ is a closed subspace 
of $L^2(\Om)$. Hence, there exists the orthogonal projection operator from $L^2(\Om)$ onto $L^2_a(\Om)$. This projection is called the Bergman projection operator of the domain $\Om$ and denoted by $\BOm$.  

The closure of a domain $\Omega$ in $\Cn$ is said to have a Stein neighborhood basis if for every neighborhood $\mathcal{U}$ of $\overline{\Omega}$ there exists a pseudoconvex domain $\mathcal{W}$ such that $\overline{\Omega}\subset\subset \mathcal{W}\subset\subset \mathcal{U}$.

The Nebenh\"ulle of $\Omega$, denoted by $\mathcal{N}(\Omega)$, is the interior of the intersection of all pseudoconvex domains 
that contain $\overline{\Omega}$. We say $\Omega$ has nontrivial Nebenh\"ulle if $\mathcal{N}(\Omega)\setminus \Omega$ has interior points.

The results in this note concern domains with non-smooth boundary therefore we have to be careful with the definition of the function space $C^{\infty}(\overline{\Omega})$. By $f\in C^{\infty}(\overline{\Omega})$ we mean for any multi-indices $\alpha,\beta$; 
$$\sup_{z\in\Omega}\left|\frac{\partial^{\alpha+\beta}}{\partial z^{\alpha}\partial\overline{z}^{\beta}}f(z)\right| \text{ is finite.}$$
Note that this condition is weaker than requiring $f$ to be smooth in a full neighborhood of $\overline{\Omega}$. We use $A^{\infty}(\Omega)$ to denote the set of holomorphic functions that are in $C^{\infty}(\overline{\Omega})$. We denote the $L^2$ Sobolev spaces of forms by $W^k_{(0,q)}(\Omega)$ for integer values of $k$.

\subsection{Background}
The Hartogs triangle is an example of a pseudoconvex domain with nontrivial Nebenh\"ulle and the worm domain in \cite{DiedFoar77a} is an example of a smooth bounded pseudoconvex domain with nontrivial Nebenh\"ulle.

It is clear that if the Nebenh\"ulle is nontrivial, then the domain cannot have a Stein neighborhood basis. On the other hand, trivial Nebenh\"ulle does not guarantee that $\Omega$ has a Stein neighborhood basis; see \cite[Proposition 1]{Sato80} for an earlier discussion and see \cite{Stensones} for a counterexample.

We refer to \cite{BedfordFornaess78, DiedFoar77a, DiederichFornaess77b, Sibony87, Sibony91, FornaessHerbig08, FornaessNagel, Harrington}
and the references within for results concerning the existence of a Stein neighborhood basis. In particular, some sufficient conditions can be listed as follows.
\begin{enumerate}
\item[(S-1)] Existence of a holomorphic vector field in a neighborhood of $b\Omega$ that is transversal to $b\Omega$ \cite{BedfordFornaess78, FornaessNagel}. 
\item[(S-2)] Smallness of a certain cohomology class \cite{BedfordFornaess78}. 
\item[(S-3)] Property $(P)$ or Property $(\tilde{P})$ \cite{Sibony87}.
\item[(S-4)] Existence of a plurisubharmonic defining function \cite{FornaessHerbig08}.
\end{enumerate}

A related problem in this context is to find sufficient conditions on a bounded domain $\Omega$ in $\Cn$ such that the $\dbar$-Neumann operator $N$, the canonical solution operator $\dbarstar N$ and the Bergman projection operator $\BOm$ of $\Omega$ are exactly regular, i.e., $N$, $\dbarstar N$ and $\BOm$ are bounded on Sobolev spaces $W^k(\Omega)$ (with correct form level) for all $k\geq 0$. Actually, on bounded pseudoconvex domains with smooth boundary if $N_{\Omega}$ is exactly regular (or compact)  then so are $\dbarstar N$ and $\BOm$ \cite{BoasStraube90, StraubeBook}.

The intriguing relationship between the problem of existence of a Stein neighborhood basis and the problem of exact regularity of $N$, $\dbarstar N$ and $\BOm$ rises to the surface when we examine the known sufficient conditions. Namely, each condition (S-1) to (S-4) has a more stringent version that implies the exact regularity of $N$, $\dbarstar N$ and $\BOm$. In particular, compare (S-1) to \cite{BoasStraube91a}, (S-2) to \cite{BoasStraube94}, (S-3) to \cite{BoasStraube91a} and \cite{HerbigMcNeal06}, and (S-4) to \cite{Catlin84}, \cite{Straube97} and \cite{McNeal04}.

In this note, we explore this relationship further on complete Hartogs domains.

\subsection{Hartogs Domains}
In this section, we go over the notion of Nebenh\"ulle on Hartogs domains (see also \cite{ZeytuncuNeben}). Let $\mathbb{D}$ denote the unit disc in $\mathbb{C}$ and let $\psi(z)$ be a continuous and bounded from below function on $\mathbb{D}$. Let us consider the domain $\Omega$ in $\mathbb{C}^2$ defined by; 
\begin{equation}\label{Hdomain}
\Omega=\left\{(z_1,z_2)\in \mathbb{C}^2~|~ z_1\in \mathbb{D}; |z_2|<e^{-\psi(z_1)}\right\}.
\end{equation}

The domain $\Omega$ is a bounded complete Hartogs domain. Moreover, it is known that (see \cite[page 129]{VlaBook}) $\Omega$ is a pseudoconvex domain if and only if $\psi(z)$ is a subharmonic function on $\mathbb{D}.$ In order to focus on pseudoconvex domains; we assume that $\psi(z)$ is a subharmonic function for the rest of the note. We further assume that $\psi(z)$ is smooth in the interior of $\mathbb{D}$.

Let $\mathcal{F}$ be the set of functions $r(z)$ where $r(z)$ is a subharmonic function on a neighborhood of $\overline{\mathbb{D}}$ such that $r(z)\leq\psi(z)$ on $\mathbb{D}$. 
We define the following two functions; 
\begin{align*}
R(z)&=\sup_{r\in\mathcal{F}}\left\{r(z) \right\},\\
R^*(z)&=\limsup_{\mathbb{D}\ni\zeta\to z}R(\zeta).
\end{align*}
Note that $R^{*}(z)$ is upper semicontinuous and subharmonic on $\mathbb{D}$.

The following proposition from \cite[Theorem 1]{Shirinbekov86} gives the description of $\mathcal{N}(\Omega)$ for $\Omega$ a complete Hartogs domain as above.
\begin{proposition}\label{description}
$\mathcal{N}(\Omega)=\left\{(z_1,z_2)\in \mathbb{C}^2~|~ z_1\in \mathbb{D}; |z_2|<e^{-R^{*}(z_1)}\right\}.$
\end{proposition}

This description does not give much information about the interior of the set difference $\mathcal{N}(\Omega)\setminus \Omega$ or the existence of a Stein neighborhood basis. The Hartogs triangle is a Hartogs domain ($\psi$ is not continuous) with nontrivial Nebenh\"ulle. The continuity assumption on $\psi$ is not enough to avoid having nontrivial Nebenh\"ulle as seen in the following example from \cite{Diederich98}.

\noindent \textbf{Example.} Consider a sequence of points in $\mathbb{D}$ that accumulates at every boundary point of $\mathbb{D}$, and let $f$ be a nonzero holomorphic function on $\mathbb{D}$ that vanishes on this sequence. 
The function defined by $\psi(z)=|f(z)|^2$ is a subharmonic function and $\Omega$, defined as above for this particular $\psi$, 
is a pseudoconvex domain. On the other hand, any pseudoconvex domain that compactly contains $\Omega$ has to contain the closure of the unit polydisc 
$\mathbb{D}\times\mathbb{D}$. Therefore, $\mathcal{N}(\Omega)\setminus \Omega$ has nonempty interior.

This example suggests we must impose additional conditions on $\psi$ or $\Omega$ to have trivial Nebenh\"ulle. The following is an example of a positive result. 

\begin{theorem}[\cite{ZeytuncuNeben}]\label{former}
Suppose $\Omega=\left\{(z_1,z_2)\in \mathbb{C}^2~|~ z_1\in \mathbb{D}; |z_2|<e^{-\psi(z_1)}\right\}$ is a smooth bounded pseudoconvex complete Hartogs domain. 
Then $\mathcal{N}(\Omega)=\Omega$, in particular $\Omega$ has trivial Nebenh\"ulle.
\end{theorem}
Note that the smoothness assumption on the domain $\Omega$ is a stronger condition than the smoothness assumption on the function $\psi(z)$.

In the this note, we prove two results on Hartogs domains $\Om$ that ensure $\mathcal{N}(\Om)=\Om$ under the assumptions of exact regularity of canonical operators.

\subsection{Results} The first result holds on any Hartogs  domain defined by \eqref{Hdomain}.

\begin{theorem}\label{main} 
Let $\Omega=\left\{(z_1,z_2)\in \mathbb{C}^2~|~ z_1\in \mathbb{D}; |z_2|<e^{-\psi(z_1)}\right\}$ where $\psi(z)$ is a smooth bounded below subharmonic function on $\mathbb{D}$. Suppose that $\dbarstar N_1$ maps $C^{\infty}_{(0,1)}(\overline{\Omega})$ to $C^{\infty}(\overline{\Omega})$ then $\mathcal{N}(\Omega)=\Omega$, in particular $\Omega$ has trivial Nebenh\"ulle.
\end{theorem}

\begin{remark}
If $N_1$ maps $C^{\infty}_{(0,1)}(\overline{\Omega})$ to $C^{\infty}_{(0,1)}(\overline{\Omega})$ then it is clear that $\dbarstar N_1$ has the desired property above since $\dbarstar$ is a first order differential operator. On the other hand, there are domains where $N_1$ fails to be regular but $\dbarstar N_1$ maps $C^{\infty}_{(0,1)}(\overline{\Omega})$ to $C^{\infty}(\overline{\Omega})$ \cite{EhsaniBidisc, EhsaniProduct}.
\end{remark}

\begin{remark}
Theorem \ref{main} is still true if the canonical solution operator $\dbarstar N_1$ is replaced by any other solution operator for the $\dbar$-problem. This resonates with the results in \cite{Chaumat} and \cite{Dufresnoy} where it is shown that existence of a Stein neighborhood basis implies existence of a regular solution operator for the $\dbar$-problem. 
\end{remark}

\begin{remark}
Theorem \ref{main} implies that for the domain constructed in the example above; the $\dbar$-Neumann operator $N_1$, the canonical solution operator $\dbarstar N_1$, or any solution operator for the $\dbar$-problem fails to be globally regular. 
\end{remark}

The second result concerns a special family of Hartogs domains. We take a bounded holomorphic function $g$ on $\mathbb{D}$ and define the following complete Hartogs domain by using this function,
\begin{equation}\label{domain}
\Omega_{g}=\left\{(z_1,z_2)\in \mathbb{C}^2~|~z_1\in \mathbb{D} \text{ and } |z_2|<|g(z_1)|\right\}.
\end{equation}
Note that here $\psi(z)=-\log |g(z)|$. Before we state the second result we observe two things.

\begin{remark}
If $g$ is constant then $\Omega_g$ is a bidisc and the closure of a bidisc admits a Stein neighborhood basis and the Bergman projection of a bidisc is exactly regular. 
\end{remark}

\begin{remark}
If $g$ vanishes at a point in $\mathbb{D}$ then $\overline{\Omega_g}$ does not admit a Stein neighborhood basis and $\BOmg$ is not bounded on Sobolev spaces. Indeed, $\Omega_g$ behaves like the famous Hartogs triangle at a zero of $g$ and these two properties fail around this point. See the last section for details. 
\end{remark}

The remaining case is when $g$ is nonconstant and has no zeros in $\mathbb{D}$. In this case, we notice that the domain does not necessarily have to admit a Stein neighborhood basis. In particular as in the example above, let us take a holomorphic function $h(z)$ on $\mathbb{D}$ such that the zero set of $h$, $\left\{z\in\mathbb{D}~|~h(z)=0\right\}$, does not have any accumulation point in the interior of $\mathbb{D}$ but it accumulates at every boundary point of $\mathbb{D}$ and also $|h(z)|<1$ on $\mathbb{D}$. We can use Blaschke products to construct such a function. Then, let $g(z)=1+h(z)$ and consider the domain $\Omega_g$. We observe that $\overline{\Omega_g}$ contains $b\mathbb{D}\times\mathbb{D}$ and therefore any pseudoconvex neighborhood of 
$\overline{\Omega_g}$ also contains $\mathbb{D}\times\mathbb{D}$. This shows, for this particular choice of $g$, the closure of the domain $\Omega_g$ does not admit a Stein neighborhood basis. 

On the other hand, if we assume the regularity of $\BOmg$ on Sobolev spaces then we get the existence of a Stein neighborhood basis.

\begin{theorem}\label{Sobolev} 
Let $g$ be a nonconstant nonvanishing bounded holomorphic function on $\mathbb{D}$.  
Suppose $\BOmg$ is bounded from $W^k(\Omega_g)$ to itself for all integers $k\geq 0$. Then the closure of $\Omega_g$ admits a Stein neighborhood basis. 
\end{theorem}

\begin{remark}
It will be clear in the proof of Theorem \ref{Sobolev} that a weaker hypothesis, namely continuity from $C^{\infty}(\overline{\Omega_g})$ to itself, would suffice. 
\end{remark}

\section{Proof of Theorem \ref{main}} 

The proof builds on the proof of Theorem \ref{former} in \cite{ZeytuncuNeben}. For the convenience of the reader, we repeat the arguments from \cite{ZeytuncuNeben} here and highlight the new ingredients in this proof. The first difference to mention is that the boundary smoothness assumption in \cite[Theorem 5]{ZeytuncuNeben} is replaced by the assumption that the canonical solution operator is globally regular.

We start as in \cite{ZeytuncuNeben} and we suppose that $\mathcal{N}(\Omega)\not=\Omega$. Then, we take a point $p=(p_1,p_2)\in \mathcal{N}(\Omega)\setminus \Omega$ and notice that actually the set difference contains more points than this singleton. Namely, by Proposition \ref{description}, $R^{*}(p_1)<\psi(p_1)$ and by semicontinuity of $R^{*}$ and continuity of $\psi$; 
there exists a neighborhood $\mathcal{U}$ of $p_1$ (compactly contained in $\mathbb{D}$) such that for all $q_1\in \mathcal{U}$, $R^{*}(q_1)<\psi(q_1)$. This neighborhood $\mathcal{U}$ guarantees that $\mathcal{N}(\Omega)$ contains a full $\mathbb{C}^2$ neighborhood of the boundary point $(p_1,e^{-\psi(p_1)}) \in b\Omega$.

Note that there exists $\delta>0$ such that $\mathcal{U}$ is contained in the set $\left\{(z_1,z_2)\in\mathbb{C}^2~|~|z_1|<1-3\delta\right\}$. Also note that we can add an appropriate function to $\psi(z)$ to construct a pseudoconvex complete Hartogs domain $\Omega_{\delta}$ with smooth boundary such that $\Omega_{\delta}\subset \Omega$ and they share the same boundary over $|z_1|<1-\delta$. We will need a smooth cut-off function $\chi_{\delta}(z_1)$ that is radially symmetric on $\mathbb{D}$, supported in $|z_1|<1-2\delta$ and identically 1 on $|z_1|<1-3\delta$.

Let $f(z_1,z_2)\in A^{\infty}(\Omega_{\delta})$ with the property that $f$ does not extend past any boundary point of $\Omega_{\delta}$; existence of such a function is guaranteed by \cite{Catlin80}. First, we extend this function to $\Omega$, 
\begin{equation*}
F(z_1,z_2)=\left\{
     \begin{array}{ll}
       f(z_1,z_2)\chi_{\delta}(z_1) &: (z_1,z_2) \in \Omega_{\delta}\\
       0 &:  (z_1,z_2)\in \Omega\setminus \Omega_{\delta}
     \end{array}
   \right. 
\end{equation*}
as a smooth function. Note that $F(z_1,z_2)\equiv f(z_1,z_2)$ for $|z_1|<1-3\delta$ and $F\in C^{\infty}({\overline{\Omega}})$.

Next, for any $q_1\in \mathcal{U}$ we define $u_{q_1}$ as 
\begin{equation*}
u_{q_1}(z_1,z_2)=\dbarstar N\left(\frac{-\dbar F (z_1,z_2)}{z_1-q_1}\right). 
\end{equation*}
The assumption on $\dbarstar N$ ensures that $u_{q_1}\in C^{\infty}(\overline{\Omega})$. By using this function we define $G_{q_1}(z_1,z_2)$ as 
\begin{equation*}
G_{q_1}(z_1,z_2)=F(z_1,z_2)+(z_1-q_1)u_{q_1}(z_1,z_2). 
\end{equation*}
Note that $G_{q_1}\in A^{\infty}(\Omega)$ and $G_{q_1}(q_1,z_2)=f(q_1,z_2)$ for all $|z_2|<e^{-\psi(q_1)}$. In the remaining part of the proof, we show that $G_{q_1}(z_1,z_2)$ extends to be a holomorphic function on the larger set $\mathcal{N}(\Omega)$ by using two lemmas from \cite{ZeytuncuNeben}. We present proofs of the lemmas for the convenience of the reader.

\begin{lemma}\label{uniform}
Suppose $p\in \mathcal{N}(\Omega)$ and $h$ is a function that is holomorphic in a neighborhood of $\overline{\Omega}$. Then $h$ has a holomorphic extension to $\mathcal{N}(\Omega)$ and $|h(p)|\leq\sup_{q\in \Omega}|h(q)|.$
\end{lemma}
\begin{proof} Note that since $h$ is holomorphic in a neighborhood of $\overline{\Omega}$, it is holomorphic on $\mathcal{N}(\Omega)$.
Next, assume the desired inequality is not true. In this case, $g(z_1,z_2)=\frac{1}{h(z_1,z_2)-h(p)}$ is a holomorphic function on some complete Hartogs domain $\Omega_1$ that compactly contains $\Omega$. 

The domain $\Omega_1$ may not be pseudoconvex but its envelope of holomorphy, denoted by $\widetilde{\Omega_1}$ that is a single-sheeted(schlicht) and complete Hartogs domain, 
is pseudoconvex (see \cite[page 183]{VlaBook}). 
Moreover, by definition any function holomorphic on $\Omega_1$ extends to a holomorphic function on the envelope of holomorphy $\widetilde{\Omega_1}$. 

In particular, $g(z_1,z_2)$ extends as holomorphic on $\widetilde{\Omega_1}$ and therefore the point $p$ can not be in $\widetilde{\Omega_1}$. But this is impossible since $p$ is a point in $\mathcal{N}(\Omega)$. \end{proof}

The next lemma, from \cite{ZeytuncuNeben}, is an approximation result that is similar to the one in \cite{BarrettFornaess}. Take a function $h$ holomorphic on $\Omega$ and expand it  as follows: $$h(z_1,z_2)=\sum_{k=0}^{\infty}a_k(z_1)z_{2}^k,$$ where each $a_k(z_1)$ is a holomorphic function on $\mathbb{D}$. Next, define the following functions (that are polynomials in $z_2$) for any $N\in\mathbb{N}$, 
\begin{equation}
\mathcal{P}_N(z_1,z_2)=\sum_{k=0}^{N}a_k\left(\frac{z_1}{1+\frac{1}{N}}\right)z_2^k.
\end{equation}
Clearly each $\mathcal{P}_N$ is holomorphic in a neighborhood of $\overline{\Omega}$. 

\begin{lemma}\label{approximation}
If $h \in A^{\infty}(\Omega)$, then the sequence of functions $\left\{\mathcal{P}_N\right\}$ converges uniformly to $h$ on $\overline{\Omega}$.
\end{lemma}

\begin{proof}
Note that $h$ easily extends to the boundary of each fiber over an interior point of $\mathbb{D}$. For $(z_1,z_2)\in \Omega$ and $k\geq 2$, by the Cauchy's inequalities;
\begin{align*}
|a_k(z_1)z_2^k|&=\left|\frac{1}{k!}\frac{(k-2)!}{2\pi i}z_2^k\int_{|\zeta|=e^{-\psi(z_1)}}\frac{\frac{\partial^2}{\partial \zeta^2}h(z_1,\zeta)}{\zeta^{k-1}}d\zeta\right|\\
&\leq \frac{1}{2\pi k(k-1)}\left(e^{-\psi(z_1)}\right)^k2\pi e^{-\psi(z_1)}\left(\sup_{\Omega}\left|\frac{\partial^2}{\partial z_2^2}h\right|\right)\frac{1}{(e^{-\psi(z_1)})^{k-1}}\\
&\leq \frac{C}{k^2}
\end{align*}
for some global constant $C$ (that depends on $\psi$ (and hence the domain $\Omega$) and the function $h$). This estimates is enough for the uniform convergence.
\end{proof}

Recall each $\mathcal{P}_N$ is holomorphic on a neighborhood of $\overline{\Omega}$ and consequently on $\mathcal{N}(\Omega)$. But by Lemma \ref{uniform}, the uniform convergence
percolates onto $\mathcal{N}(\Omega)$ and therefore any function in $A^{\infty}(\Omega)$ extends to a holomorphic function on $\mathcal{N}(\Omega)$. 

When we apply the argument above to $G_{q_1}$, we note that $G_{q_1}$ is a holomorphic function on $\mathcal{N}(\Omega)$. In particular, $G_{q_1}(q_1,z_2)=f(q_1,z_2)$ extends to a larger disc $|z_2|<e^{-R^*(q_1)}$.

Recall $q_1\in \mathcal{U}$ is an arbitrary point, and therefore we observe that for any point $q_1\in \mathcal{U}$; the holomorphic function $f(q_1,z_2)$ extends from the disc $|z_2|<e^{-\psi(q_1)}$ to a strictly larger disc $|z_2|<e^{-R^*(q_1)}$. This implies $f(z_1,z_2)$ extends a holomorphic function (by the joint analyticity lemma of Hartogs) past the boundary point $(p_1,e^{\psi(p_1)})$ but this is a contradiction since $f$ is assumed to be non-extendable. Hence we conclude that indeed, $$\mathcal{N}(\Omega)=\Omega.$$

\section{Proof of Theorem \ref{Sobolev}}

Let $\chi(z)$ be a nonzero compactly supported smooth radial function on $\mathbb{D}$. The key observation is the following lemma.

\begin{lemma}\label{rep}
$\BOmg\left(\frac{\chi(z_1)}{g(z_1)}\right)(z_1,z_2)=\frac{c}{g(z_1)}$
for some nonzero constant $c$. 
\end{lemma}

\begin{proof} 
Take a holomorphic function $h(z_1,z_2)$ that is in $L^2(\Omega_g)$ and consider the following two inner products $\Omega_g$.

\begin{align*}
\left<\frac{\chi(z_1)}{g(z_1)},h(z_1,z_2)\right>&=\pi \int_{\mathbb{D}} \frac{\chi(z_1)}{g(z_1)}\overline{h(z_1,0)}|g(z_1)|^2dA(z_1)\\
&=\pi \int_{\mathbb{D}} \chi(z_1)\overline{h(z_1,0)g(z_1)} dA(z_1)\\
&=c_1\overline{h(0,0)g(0)}\\
\left<\frac{1}{g(z_1)},h(z_1,z_2)\right>&=\pi \int_{\mathbb{D}} \frac{1}{g(z_1)}\overline{h(z_1,0)}|g(z_1)|^2dA(z_1)\\ 
&=\pi \int_{\mathbb{D}} \overline{h(z_1,0)g(z_1)} dA(z_1)\\
&=c_2\overline{h(0,0)g(0)}
\end{align*}
Note that $c_1,c_2$ are nonzero and we have $\frac{\chi(z_1)}{g(z_1)}-\frac{c_1}{c_2g(z_1)}$ is perpendicular to all square integrable holomorphic functions on $\Omega_g$.
\end{proof}

It is clear that $\frac{\chi(z_1)}{g(z_1)} \in W^k(\Omega_g)$ for all $k\geq 0$ and the regularity assumption implies that $\frac{1}{g(z_1)}=\frac{c_2}{c_1}\BOmg\left(\frac{\chi}{g}\right) \in W^k(\Omega_g)$ for all $k\geq 0$. This means for any $k\geq 0$,
\begin{equation}\label{derivatives}
\int_{\mathbb{D}}\left| \frac{\partial^k}{\partial z^k}\left(\frac{1}{g(z)}\right)\right|^2~|g(z)|^2dA(z)
\end{equation}
is finite. We use these estimates to conclude that $g(z)$ is uniformly continuous on $\mathbb{D}$.

First, we show that $|g(z)|$ does not decay too fast. Note that there exists a holomorphic function $h(z)$ on $\mathbb{D}$ such that $$g(z)=e^{h(z)}.$$
The case $k=1$ in \eqref{derivatives} gives $h'(z)$ is square integrable on $\mathbb{D}$ and if we let $h(z)=\sum_{n=0}^{\infty}h_nz^n$ we obtain the following.
\begin{align*}
|h(z)|^2&=\left|\sum_{n=0}^{\infty}h_nz^n\right|^2\\
&\leq \left(\sum_{n=1}^{\infty}|h_n|^2n\right)\left(\sum_{n=1}^{\infty}\frac{|z|^{2n}}{n}\right) + |h(0)|^2\\
&\lesssim ||h'||^2\log\left(\frac{1}{1-|z|^2}\right) + |h(0)|^2.
\end{align*}
This implies that there exists $M>0$ such that,
\begin{align*}
|\Re h(z)|&\leq|h(z)|\\
&\leq M\log\left(\frac{1}{1-|z|^2}\right) + |h(0)|.
\end{align*}

Without loss of generality, we can assume $|g(z)|<1$ and this gives us $\Re h(z)<0$ on $\mathbb{D}$. On the other hand, by using the previous estimate we get 
\begin{equation*}
\Re h(z)\geq -|h(0)|-M\log\left(\frac{1}{1-|z|^2}\right).
\end{equation*}
This gives us a lower bound on $g(z)$,
\begin{equation*}
|g(z)|=e^{\Re h(z)}\geq C(1-|z|^2)^M.
\end{equation*}

We modify \eqref{derivatives} by using this lower bound and  we get for any $k\geq 0$,
\begin{equation*}
\int_{\mathbb{D}}\left| \frac{\partial^k}{\partial z^k}\left(\frac{1}{g(z)}\right)\right|^2~(1-|z|^2)^{2M}dA(z)
\end{equation*}
is finite.

Consider the Taylor series expansion of $\frac{1}{g(z)}=\sum_{n=0}^{\infty}b_nz^n$. By using the orthogonality and Beta functions we get
\begin{align*}
\int_{\mathbb{D}}\left| \frac{\partial^k}{\partial z^k}\left(\frac{1}{g(z)}\right)\right|^2~(1-|z|^2)^{2M}dA(z)&=\int_{\mathbb{D}}\left| \sum_{n=0}^{\infty}b_{n+k}\frac{(n+k)!}{n!}z^{n}\right|^2~(1-|z|^2)^{2M}dA(z)\\
&=\sum_{n=0}^{\infty}|b_{n+k}|^2\frac{((n+k)!)^2}{(n!)^2}\int_{\mathbb{D}}|z|^{2n}(1-|z|^2)^{2M}dA(z)\\
&\geq C_k \sum_{n=0}^{\infty}|b_{n+k}|^2n^{2k}\frac{1}{n^{2M+1}}
\end{align*}
for any $k\geq 0$. Therefore, $\lim_{n\to \infty}n^j|b_n|=0$ for any $j\geq0$.

In particular, this implies $\frac{1}{g(z)}$ is bounded on $\mathbb{D}$ and hence $g(z)$ is bounded from below. We also conclude that
$\frac{1}{g(z)}\in W^2(\mathbb{D})$. These last two implications indicate that $g(z)$ is uniformly continuous on $\mathbb{D}$. 
Furthermore, for any $\epsilon>0$, there exists $\delta(\epsilon)=\delta>0$ such that 
\begin{equation*}
|g(z)|< (1+\epsilon)\left|g\left(\frac{z}{1+\delta}\right)\right|
\end{equation*}
for all $z\in\mathbb{D}$.

Now let us define the following domains for any $\epsilon>0$
\begin{equation*}
\Omega_{\epsilon}=\left\{(z_1,z_2)\in \mathbb{C}^2~|~ |z_1|< 1+\delta \text{ and } |z_2|<(1+\epsilon)\left|g\left(\frac{z_1}{1+\delta}\right)\right|\right\}.
\end{equation*}

It is clear that each $\Omega_{\epsilon}$ is a pseudoconvex domain and each one compactly contains $\Omega_g$. Furthermore, 
\begin{equation*}
\bigcap_{\epsilon>0}\Omega_{\epsilon}=\overline{\Omega_g}~.
\end{equation*}
This shows the existence of a Stein neighborhood basis and finishes the proof of Theorem \ref{Sobolev}.

\subsection*{Zeros in $\mathbb{D}$}

Suppose $g(z)$ is a nonconstant holomorphic function on $\mathbb{D}$ and suppose $g(z_0)=0$ for some $z_0\in\mathbb{D}$. Then $\Omega_g$ behaves like the 
Hartogs triangle around this point and therefore the closure does not admit a Stein neighborhood basis. 

On the other hand, we note that $\BOmg$ is not exactly regular. Let's suppose that $z_0=0$. For the general case, we can use an automorphism of the unit disc and the fact that an automorphism of the disc preserves the regularity properties of $\BOmg$. We factor $g(z)$ as $z^jh(z)$ for some integer $j$ and some holomorphic function $h(z)$ that is zero-free in a neighborhood of $0$. We also take a radially symmetric real cut-off function $\chi$ that is supported in a small enough
neighborhood of $0$. We consider the function $$\chi(z_1)\frac{g(z_1)\overline{g(z_1)}^2}{|h(z_1)|^4}=\chi(z_1)\frac{|z_1|^{2j}\overline{z_1}^j}{h(z_1)}.$$ 
This function belongs to $W^k(\Omega_g)$ for any integer $k\geq 0$. However, Lemma \ref{rep} implies that
\begin{equation*}
\BOmg\left(\chi(z_1)\frac{g(z_1)\overline{g(z_1)}^2}{|h(z_1)|^4}\right)=\frac{c}{g(z_1)} 
\end{equation*}
for some constant $c$. It is clear that $\frac{c}{g(z_1)}\not\in W^1(\Omega_g)$. Therefore, $\BOmg$ is not exactly regular.

\section*{Acknowledgments} I'd like to thank H.P. Boas, S. \c{S}ahuto\~{g}lu and S. Ravisankar for useful remarks on this paper. I'd also like to thank the anonymous referee for valuable comments that significantly improved the overall quality of the paper.

\vskip 1cm
\bibliographystyle{amsalpha}
\bibliography{Nebenhulle}

\vskip 1 cm
\end{document}